\documentclass[12pt]{article}
\usepackage{amsmath}
\usepackage{amsfonts}
\usepackage{amssymb}
\usepackage{amscd}
\usepackage{amsthm}
\usepackage{color}
\usepackage[cp1251]{inputenc}
\usepackage[russian]{babel}
\usepackage{tikz-cd}

\makeatletter
\def\smalloverbrace#1{\mathop{\vbox{\m@th\ialign{##\crcr\noalign{\kern3\p@}%
				\tiny\downbracefill\crcr\noalign{\kern3\p@\nointerlineskip}%
				$\hfil\displaystyle{#1}\hfil$\crcr}}}\limits}
\makeatother

\numberwithin{equation}{section}

\theoremstyle{plain}
\newtheorem{theorem}{Теорема}[section]
\newtheorem{lemma}[theorem]{Лемма}
\newtheorem{proposition}[theorem]{Предложение}
\newtheorem{corollary}[theorem]{Следствие}

\theoremstyle{definition}
\newtheorem{definition}[theorem]{Определение}
\newtheorem{remark}[theorem]{Замечание}

\def \Coker {\operatorname{Coker}}

\newcommand{\Ab}{{\mathcal A}b}

\title{Еще раз об аналоге одной теоремы Воеводского}

\author{И. А. Панин, Д.Н. Тюрин}

\begin{document}

	\maketitle
\begin{abstract}
Пусть $F$ является $\mathbb{A}^{1}$-инвариантным квази-стабильным $\mathbb{Z}F_{\ast}$-предпучком. Тогда его пучковизация по Зарисскому $F_{Zar}$ совпадает с пучковизацией по Нисневичу $F_{Nis}$. Кроме того, для любой $k$-гладкой схемы $X\in Sm/k$ имеют место равенства $H^{n}_{Zar}(X, F_{Zar})=H^{n}_{Nis}(X,F_{Nis})$. 
\end{abstract}

\section{Введение}

Одним из основополагающих результатов статьи \textit{``Cohomological theory of presheaves with transfers''} В.А. Воеводского, 
на котором основано построение триангулированной категории мотивов Воеводского, является теорема, известная теперь как теорема Воеводского. 
В ней утверждается, что если поле $k$ совершенно, то для любого $\mathbb{A}^{1}$-инвариантного предпучка с трансферами $F$ на категории $Sm/k$ выполняются следующие утверждения:
\begin{itemize}
	\item
	\cite[Предложение 5.5]{V1} пучок $F_{Zar}$ совпадает с пучком $F_{Nis}$ и имеет естественную структуру $\mathbb{A}^{1}$-инвариантного предпучка с трансферами;
	\item
	\cite[Теорема 5.6]{V1} предпучки $X\mapsto H^{n}_{Zar}(X,F_{Zar})=H^{n}_{Nis}(-,F_{Nis})$ гомотопически инвариантны (и имеют естественную структуру предпучков с трансферами);
	\item
	\cite[Теорема 5.7]{V1} для любого $X\in Sm/k$ и для любого $n\ge 0$ имеет место равенство $H^{n}_{Zar}(X,F_{Zar})=H^{n}_{Nis}(-,F_{Nis})$

\end{itemize}
(См. также~\cite[Теорема 3.1.12]{V2}). Неформально можно сказать, что $\mathbb{A}^{1}$-инвариантные пучки Нисневича с трансферами играют в  
триангулированной категории мотивов Воеводского роль своеобразных элементарных ``кирпичиков'', на которые развинчивается любой обьект категории. Позднее, в  в своих основополагающих заметках  ``Notes on framed correspondences'' \cite{V3}. Воеводский ввел понятие предпучков (множеств) с оснащенными трансферами (фрэйм трансферами), а также предположил, что для стабильной мотивной гомотопической категории $SH(k)$ роль, аналогичную вышеописанной, играют аддитивные квази-стабильные $\mathbb{A}^{1}$-инвариантные пучки Нисневича абелевых групп с оснащенными трансферами. Правильность этого предположения в числе прочего была доказана Гаркушей и Паниным в статье ``Framed motives of algebraic varieties (after V. Voevodsky)'' \cite{GP2}. 
Для этого они доказали в своей в фундаментальной статье \textit{``Homotopy invariant presheaves with framed transfers''} \cite{GP1} почти полный аналог сформулированной выше теоремы Воеводского (вынеся за скобки топологию Зариского). При этом была введена категория
 $\mathbb{Z}F_{\ast}$-линейных оснащенных соответствий и показано, что аддитивные (пред)пучки Нисневича абелевых групп с оснащенными трансферами совпадают с $\mathbb{Z}F_{\ast}$-предпучками абелевых групп. Тем самым аддитивные $\mathbb{A}^{1}$-инвариантные
квази-стабильные (пред)пучки Нисневича абелевых групп с оснащенными трансферами можно рассматривать как $\mathbb{A}^{1}$-инвариантные,
квази-стабильные $\mathbb{Z}F_{\ast}$-(пред)пучки абелевых групп.

Основная теорема~\cite{GP1} утверждает, что для любого $\mathbb{A}^{1}$-инвариантного, квази-стабильного $\mathbb{Z}F_{\ast}$-предпучка абелевых групп $F$ на $Sm/k$ 
ассоциированный с ним пучок Нисневича $F_{Nis}$ и все предпучки его когомологий  $H^{n}_{Nis}(-,F_{Nis})$ снабжены канонической структурой $\mathbb{Z}F_{\ast}$-предпучков, 
а кроме того --- также являются $\mathbb{A}^{1}$-инвариантными и квази-стабильными (подробнее см. \cite[Теорема 1.1, Следствие 2.17]{GP1}). 
Напомним, что эта теорема доказана ими в предположении, что характеристика $k$ не равна 2. В характеристике 2 указанный результат доказан
в статье \cite{DP}.

Цель настоящей статьи - вспомнить о топологии Зариского, чтобы сделать основную теорему Гаркуши--Панина полным аналогом теоремы Воеводского. 
А именно, мы докажем что, если $F$ является $\mathbb{A}^{1}$-инвариантным  квази-стабильным $\mathbb{Z}F_{\ast}$-предпучком абелевых групп, то:
\begin{itemize}
\item пучок $F_{Zar}$ совпадает с пучком $F_{Nis}$ (в частности, оба они являются $\mathbb{A}^{1}$-инвариантными
квази-стабильными $\mathbb{Z}F_{\ast}$-пучками абелевых групп);
\item для любого $X\in Sm/k$ и для любого $n\ge 0$ имеет место изоморфизм $H^{n}_{Zar}(X,F_{Zar})\cong H^{n}_{Nis}(X,F_{Nis})$.
\end{itemize}
Отметим, что, поскольку по основной теореме Гаркуши--Панина, предпучки вида $X\mapsto H^{n}_{Nis}(X,F_{Nis})$ являются $\mathbb{A}^{1}$-инвариантными
квази-стабильными $\mathbb{Z}F_{\ast}$-предпучками абелевых групп, то также  имеет место следствие:
\begin{itemize} 
\item (предпучки вида $X\mapsto H^{n}_{Zar}(X,F_{Zar})$ являются $\mathbb{A}^{1}$-инвариантными
квази-стабильными $\mathbb{Z}F_{\ast}$-предпучками абелевых групп
\end{itemize}


Настоящая статья является прямым продолжением работы \textit{``Of a certain analogue of Voevodsky theorem''} ~\cite{D}, так же посвященной аналогу теоремы Воеводского для случая $\mathbb{Z}F_{\ast}$-предпучков. В частности, для доказательства нам потребуется утверждение о том, что пучок $F_{Nis}$ обладает на $X$ вялой резольвентой (Герстена):
$$
0\to F_{Nis}\to\underset{x\in X^{(0)}}{\bigoplus}(i_{x})_{\ast}(F_{Nis})\to\dots\to\underset{x\in X^{(m)}}{\bigoplus}(i_{x})_{\ast}((F_{Nis})_{-m})\to 0\,,
$$
доказанное в ~\cite[Теорема 6.4]{D}.

Исследование финансировалось в рамках Программы фундаментальных исследований НИУ ВШЭ.

\section{Определения и вспомогательные результаты}

Пусть $k$ --- совершенное поле. Через $Sm/k$ мы будем обозначать категорию $k$-гладких схем конечного типа.  Через $Sm'/k$ мы будем обозначать категорию существенно $k$-гладких схем. Если $G$ является предпучком абелевых групп на $Sm/k$, мы, для краткости, будем обозначать тем же символом его прямой образ относительно естественного вложения категорий $Sm/k\hookrightarrow Sm'/k$.

В первую очередь мы напомним некоторые определения связанные с категорией  $\mathbb{Z}F_{\ast}$-предпучков на $Sm/k$:

\begin{itemize}
	\item 
	 Через $\mathbb{Z}F_{\ast}$ мы обозначаем аддитивную категорию, объекты которой --- это в точности объекты категории $Sm/k$, а морфизмы имеют вид
	$$
	Hom_{\mathbb{Z}F_{\ast}}(Y,X)=\underset{n\geqslant0}{\bigoplus}\mathbb{Z}F_{n}(Y,X)\,.
	$$
	Через $\mathbb{Z}F_{n}(Y,X)$ здесь обозначается фактор-группа свободной абалевой группы $\mathbb{Z}[Fr_{n}(Y,X)]$, порожденной всеми фрейм-соответствиями уровня $n$ из $Y$ в $X$, по подгруппе, порожденной всеми элементами вида
	$$
	(Z\sqcup Z',V,\varphi,g)-(Z,V\setminus Z',\varphi|_{V\setminus Z'},g|_{V\setminus Z'})-(Z',V\setminus Z,\varphi|_{V\setminus Z},g|_{V\setminus Z})
	$$
	(подробнее смотри в [ссылка на статью]);
	\item
	$\mathbb{Z}F_{\ast}$-предпучком абелевых групп называется аддитивный контрвариантный функтор из $\mathbb{Z}F_{\ast}$ в категорию $Ab$ абелевых групп. $\mathbb{Z}F_{\ast}$-предпучок называется $\mathbb{Z}F_{\ast}$-пучком (в соответствующей топологии), если он является пучком относительно $Sm/k$ (в соответствующей топологии);
	\item 
    $\mathbb{Z}F_{\ast}$-предпучок называется $\mathbb{A}^{1}$-инвариантным, если для любого $X\in Sm/k$ проекция $X\times\mathbb{A}^{1}\to X$ индуцирует изоморфизм $F(X)\to F(X\times\mathbb{A}^{1})$;
	\item
    $\mathbb{Z}F_{\ast}$-предпучок называется квази-стабильным, если для любого $X\in Sm/k$ фрейм-соответствие $\sigma_{X}:=(X\times 0,X\times\mathbb{A}^{1},t,pr_{X})\in Fr_{1}(X,X)$ индуцирует изоморфизм $\sigma^{\ast}_{X}:F(X)\to F(X)$. Отметим, что для любого $\psi\in\mathbb{Z}F_{\ast}(X,Y)$ имеет место равенство
    $\psi\circ\sigma_{X}=\sigma_{Y}\circ\psi$.
\end{itemize}

Мы также будем говорить, что $\mathbb{Z}F_{\ast}$-пучок $\mathbb{A}^{1}$-инвариантен (квази-стабилен) если он $\mathbb{A}^{1}$-инвариантен (квази-стабилен) как $\mathbb{Z}F_{\ast}$-предпучок.

Для дальнейшего удобства введем следующее обозначение:
\begin{definition}
Будем говорить, что $\mathbb{Z}F_{\ast}$-предпучок $F$ удовлетворяет свойству $(\ast)$ (является $(\ast)$-предпучком), если он является $\mathbb{A}^{1}$-инвариантным и квази-стабильным.
\end{definition}

Одним из основных результатов, на которые мы будем опираться в наших рассуждениях является следующая теорема:

\begin{theorem}\label{the:PGmain}
	~\cite[Лемма 4.5]{V3}, \cite[Теорема 1.1, Следствие 2.17]{GP1} Для любого $(\ast)$-предпучка $F$ на соответствующей пучковизации по Нисневичу $F_{Nis}$ существует единственная структура $\mathbb{Z}F_{\ast}$-предпучка согласованная с естественным морфизмом $F\to F_{Nis}$. Кроме того $F_{Nis}$, а также все соотвествующие предпучки когомологий $X\to H^{n}_{Nis}(X,F_{Nis})$ тоже являются $(\ast)$-предпучками.
\end{theorem}

Приведем также несколько вспомогательных утверждений: 

\begin{theorem}\label{the:PGinj} \cite[Теорема 3.15(3)]{GP1} Пусть $X$ -- неприводимое гладкое многообразие над $k$, $x\in X$ и $U:=Spec(\mathcal{O}_{X,x})$.Тогда для любого $(\ast)$-предпучка $F$ гомоморфизм $F(U)\to F(Spec(k(X))$ индуцированный каноническим морфизмом $Spec(k(X))\to U$ является иньективным.
\end{theorem}

\begin{corollary}\label{cor:inj}
	Пусть $X$ -- неприводимая гладкая $k$-схема, $x\in X$ и $U:=Spec(\mathcal{O}_{X,x})$.Тогда для любого $(\ast)$-предпучка $F$ естественный гомоморфизм $F(U)\to F_{Nis}(U)$ является иньективным.
\end{corollary}
\begin{proof} Рассмотрим коммутативную диаграмму
	$$
	\begin{CD}\label{com:1}
		F(U) @>>> F(Spec(k(X)))\\
		@VVV @VVV\\
		F_{Nis}(U) @>>> F_{Nis}(Spec(k(X)))
	\end{CD}
	$$
	Утверждение следует из того, что стрелка $F(U)\to F(Spec(k(X)))$ иньективна согласно теореме~\ref{the:PGinj}, и что морфизм  $F(Spec(k(X)))\to F_{Nis}(Spec(k(X)))$ является тождественным.
	\end{proof}
	
	Пусть теперь $X\in Sm'/k$, $Y\in Sm/k$. Через $\overline{\mathbb{Z}F}_{\ast}(X,Y)$ мы будем обозначать группу
	$$
	\Coker\Big[\mathbb{Z}F_{\ast}(\mathbb{A}^{1}\times X,Y)\xrightarrow{i^{\ast}_{0}-i^{\ast}_{1}}\mathbb{Z}F_{\ast}(X,Y)\Big]\,,
	$$

где через $i_{0,1}$ обозначаются, соотвественно, вложения $\{0\},\{1\}\hookrightarrow\mathbb{A}^{1}$. Соответствующий класс фрейм-соответствия $\phi\in \mathbb{Z}F_{\ast}(X,Y)$ мы будем обозначать через $[\phi]$.

Следующая геометрическая лемма будет играть ключевую роль в нашем доказательстве.	 
\begin{lemma}\label{lemm:geom} \cite[Утверждение 9.9]{GP1}
	Пусть $X$ -- неприводимое гладкое многообразие над $k$, $x\in X$ и $U:=Spec(\mathcal{O}_{X,x})$. Обозначим естественное вложение $U\hookrightarrow X$ через $can$. Пусть также $D\subset X$ --- замкнутое собственное подмножество в $X$ и $j:X-D\hookrightarrow X$ --- соответствующее открытое вложение. Тогда существует такое натуральное число $N$ и такое $\phi\in \mathbb{Z}F_{N}(U,X-D)$, что в $\overline{\mathbb{Z}F}_{Т}(U,X)$ имеет место равенство
\begin{equation}\label{eq:geom}
	[j]\circ[\phi]=[\sigma^{N}_{X}]\circ[can]\,.
\end{equation}	

	\end{lemma}

Мы закончим этот параграф формулировкой основного результата нашей статьи:
\begin{theorem}\label{the:main}
	Пусть $F$ -- $(\ast)$-предпучок. Тогда на категории $Sm/k$ имеет место равенство пучков Зарисского $F_{Nis}=F_{Zar}$. Более того, для любой $k$-гладкой схемы $X\in Sm/k$ и для любого $n>0$ имеет место равенство $H^{n}_{Zar}(X,F_{Zar})=H^{n}_{Nis}(X,F_{Nis})$.
\end{theorem}	

В частности, с учетом теоремы~\ref{the:PGmain}, отсюда сразу же следует, что пучок $F_{Zar}$ и все предпучки его когомологий $H^{n}(-,F_{Zar})$ являются $(\ast)$-предпучками.

\section{Случай $F_{Zar}$}

Рассмотрим естественный гомоморфизм пучков по Зарисскому $F_{Zar}\to F_{Nis}$. Для доказательства того, что это изоморфизм нам достаточно проверить его на ростках вида  $U:=Spec(\mathcal{O}_{X,x})$, где $X$ --- неприводимая $k$-гладкая схема и $x$ --- некоторая точка в $X$.
\begin{proposition}\label{prop:main}
	Пусть $F$ --- $(\ast)$-предпучок, $a:F\to F_{Nis}$ --- естественный гомоморфизм пучкования, $X\in Sm/k$ --- неприводимая гладкая $k$-схема и $x\in X$. Обозначим $Spec(\mathcal{O}_{X,x})$ через $U$. Тогда соответствующий гомоморфизм $a_{U}:F(U)=F_{Zar}(U)\to F_{Nis}(U)$ является изоморфизмом. 
\end{proposition}

\begin{remark}
	В дальнейшем, ради удобства, для любого фрейм-соответствия $\theta$ лежащего в $\mathbb{Z}F_{\ast}(X,Y)$ мы будем обозначать соответствующий морфизм обратного образа $\theta^{\ast}:F(Y)\to F(X)$ для предпучка $F$ и $\widetilde{\theta}^{\ast}:F_{Nis}(Y)\to F_{Nis}(X)$ для пучка $F_{Nis}$. В частности, из того, что $a$ является морфизмом $\mathbb{Z}F_{\ast}$-предпучков по теореме~\ref{the:PGmain}, следует, что соответствующая диаграмма
	$$
	\begin{CD}\label{com:1}
		F(Y) @>\theta^{\ast}>> F(X)\\
		@Va_{Y}VV @Va_{X}VV\\
		F_{Nis}(Y) @>\widetilde{\theta}^{\ast}>> F_{Nis}(X)
	\end{CD}
	$$
	является коммутативной.
\end{remark}
	
Отдельно рассмотрим случай, $\theta=\eta_{Y}$, где $\eta_{Y}:Spec(k(Y))\to Y$ --- вложение общей точки в $Y$. Так как  $a_{k(Y)}$ тождественно,  в $F(Spec(k(Y)))=F_{Nis}(Spec(k(Y)))$ имеет место формула
\begin{equation}\label{eq:eq1}
	\eta^{\ast}_{Y}=\widetilde{\eta}^{\ast}_{Y}\circ a_{Y}\,.
\end{equation}	

Перейдем теперь к доказательству предложения. С учетом следствия~\ref{cor:inj} достаточно показать, что $a_{U}$ сюрьективен. Пусть $\alpha\in F_{Nis}(U)$. Обрезав, по необходимости $X$, с сохранением точки $x$ мы можем считать, что существует глобальное сечение $\widehat{\alpha}\in F_{Nis}(X)$, такое, что $\widetilde{can}^{\ast}(\widehat{\alpha})=\alpha$ . Более того, так как 
$F(k(X))=F_{Nis}(k(X))$, существует такое замкнутое подмножество $D\subset X$ и такой элемент $\widehat{\alpha}'\in F(X-D)$, что $\widehat{\alpha}|_{k(X)}=\widehat{\alpha}'|_{k(X)}$. Другими словами, для коммутативной диаграммы
\begin{equation}\label{eq:com1}
\begin{tikzcd}
	U \arrow{r}{can}  & X   & \arrow{l}[swap]{j}  X-D \\	
	&    Spec(k(X)) \arrow{ul}{\eta_{U}} \arrow{u}{\eta_{X}} \arrow{ur}[swap]{\eta_{X-D}}     &
\end{tikzcd}
\end{equation}
имеет место равенство
\begin{equation}\label{eq:eq2}
	\widetilde{\eta}^{\ast}_{U}(\widetilde{can}^{\ast}(\widehat{\alpha}))=\eta^{\ast}_{X-D}(\widehat{\alpha}')\,.
	\end{equation}
	
	\begin{lemma}\label{lemm:key}
		В $F_{Nis}(X-D)$ имеет место равенство
		\begin{equation}\label{eq:eq3}
		a_{X-D}(\widehat{\alpha}')=\widetilde{j}^{\ast}(\widehat{\alpha})\,.
			\end{equation}
	\end{lemma}
	\begin{proof}	
	В силу коммутативности диаграммы~\ref{eq:com1} и формулы~\ref{eq:eq2} получаем
	$$
	\eta^{\ast}_{X-D}(\widehat{\alpha}')=\widetilde{\eta}^{\ast}_{U}(\widetilde{can}^{\ast}(\widehat{\alpha}))=\widetilde{\eta}^{\ast}_{X-D}(\widetilde{j}^{\ast}(\widehat{\alpha}))
	$$
	
	В то же время из формулы~\ref{eq:eq1} для случая $Y=X-D$ следует, что 
	$$
	\eta^{\ast}_{X-D}(\widehat{\alpha}')=\widetilde{\eta}^{\ast}_{X-D}(a_{X-D}(\widehat{\alpha}'))\,.
	$$
	
	Таким образом, $a_{X-D}(\widehat{\alpha}')$ и $\widetilde{j}^{\ast}(\widehat{\alpha})$ совпадают по модулю отображения 
	$$
	\widetilde{\eta}^{\ast}_{X-D}:F_{Nis}(X-D)\longrightarrow F_{Nis}(Spec(k(X)))=F(Spec(k(X)))\,.
	$$
	
	Покажем, что это отображение на самом деле иньективно. Действительно, так как $F_{Nis}$ является также пучком по Зарисскому, любой его элемент представляется в виде согласованного набора сечений из $\{F_{Nis}(Spec(\mathcal{O}_{X-D,y}))\}_{y\in X-D}$. Но так как по теореме~\ref{the:PGmain} пучок $F_{Nis}$ также остается $(\ast)$-предпучком, то для любого $y\in X-D$ соответствующее естественное отображение 
	$$
	F_{Nis}(Spec(\mathcal{O}_{X-D,y}))\to F_{Nis}(Spec(k(X)))
	$$
	является иньективным по теореме~\ref{the:PGinj}. Следовательно, иньективным является и само отображение $\widetilde{\eta}^{\ast}_{X-D}$. Лемма доказана.
		\end{proof}
		
	Вернемся к доказательству предложения. Пусть $\phi$ --- элемент из $\mathbb{Z}F_{N}(U,X-D)$, удовлетворяющий условиям леммы~\ref{lemm:geom}. Так как $F$ и $F_{Nis}$ являются $(\ast)$-предпучками мы можем считать, что выполняется равенство $\phi^{\ast}\circ j^{\ast}=can^{\ast}$ (соответственно, $\widetilde{\phi}^{\ast}\circ \widetilde{j}^{\ast}=\widetilde{can}^{\ast}$). Теперь достаточно доказать, что в  $F_{Nis}(Spec(k(X))=F(Spec(k(X))$ имеет место формула

\begin{equation}\label{eq:eq4}
\eta^{\ast}_{U}(\phi^{\ast}(\widehat{\alpha}'))=\eta^{\ast}_{X-D}(\widehat{\alpha}')\,.
\end{equation}
Действительно, с учетом формул~\ref{eq:eq1}, \ref{eq:eq2} и~\ref{eq:eq4} мы получаем равенства 
$$
\widetilde{\eta}^{\ast}_{U}(\widetilde{can}^{\ast}(\widehat{\alpha}))=\eta^{\ast}_{U}(\phi^{\ast}(\widehat{\alpha}'))=\widetilde{\eta}^{\ast}_{U}( a_{U}(\phi^{\ast}(\widehat{\alpha}'))
$$
Но в силу того, что $F_{Nis}$ остается $(\ast)$-предпучком, по теореме~\ref{the:PGinj} морфизм
$$
\widetilde{\eta}^{\ast}_{U}:F_{Nis}(U)\longrightarrow F_{Nis}(Spec(k(X)))=F(Spec(k(X)))
$$
является иньективным. Таким образом, $\alpha=\widetilde{can}^{\ast}(\widehat{\alpha})$ будет совпадать с образом $\phi^{\ast}(\widehat{\alpha}')$ относительно $a_{U}$.

Докажем формулу~\ref{eq:eq4}. Поскольку гомоморфизм $a_{k(X)}$ тождественнен, достаточно проверить равенство
$$
a_{k(X)}(\eta^{\ast}_{U}(\phi^{\ast}(\widehat{\alpha}')))=a_{k(X)}(\eta^{\ast}_{X-D}(\widehat{\alpha}'))\,.
$$

Т.к. $a$ --- гомоморфизм предпучков с трансферами, мы получаем эквивалентное равенство	
$$
\widetilde{\eta}^{\ast}_{U}(\widetilde{\phi}^{\ast}(a_{X-D}(\widehat{\alpha}')))=\widetilde{\eta}^{\ast}_{X-D}(a_{X-D}(\widehat{\alpha}'))\,.
$$
Подставляя $\widetilde{j}^{\ast}(\widehat{\alpha})$ вместо $a_{X-D}(\widehat{\alpha}')$ согласно формуле~\ref{eq:eq3}, получаем
$$
\widetilde{\eta}^{\ast}_{U}(\widetilde{\phi}^{\ast}(\widetilde{j}^{\ast}(\widehat{\alpha}))=\widetilde{\eta}^{\ast}_{X-D}(\widetilde{j}^{\ast}(\widehat{\alpha}))\,.
$$
Из коммутативности диаграммы~\ref{eq:com1} следует что $\widetilde{\eta}^{\ast}_{X-D}(\widetilde{j}^{\ast}(\widehat{\alpha}))=\widetilde{\eta}^{\ast}_{U}(\widetilde{can}^{\ast}(\widehat{\alpha})))$. Так как по условию $\widetilde{\phi}^{\ast}\circ \widetilde{j}^{\ast}$ совпадает с $\widetilde{can}^{\ast}$, равенство~\ref{eq:eq4}, а следовательно и предложение~\ref{prop:main} доказаны.

	\section{Случай когомологий}
 Рассмотрим контрвариантный функтор:	
	$$
	H^{\bullet}: SmOp/k\to Gr\Ab\,,\,\, (X,X-Z)\mapsto \underset{n\in\mathbb{N}}{\bigoplus}(H^{n}_{Nis})_{Z}(X,F_{Nis})\,.
	$$

Так как по теореме~\ref{the:PGmain} все предпучки когомологий вида	$H^{n}_{Nis}(-,F_{Nis})$ являются $\mathbb{A}^{1}$-инвариантными, этот функтор является теорией когомологий в смысле Панина--Смирнова (см.~\cite{PS}). Поэтому в соответствии с ~\cite[Следствие 9.2]{P1} для любого неприводимого $X\in Sm/k$ размерности $m$ соответствующий пучок Нисневича $F_{Nis}$ (являющийся, в частности, пучком Зарисского) имеет на $X_{Zar}$ вялую резольвенту 
$$
0\to F_{Nis}(-) \to\eta_{\ast}(F_{Nis}(\eta))\to \underset{x\in (-)^{(1)}}{\bigoplus}(H^{1}_{Nis})_{x}(-,F_{Nis})\to\dots
$$
$$
\dots\to  \underset{x\in (-)^{(m)}}{\bigoplus}(H^{m}_{Nis})_{x}(-,F_{Nis})\to 0\,
$$

Пусть теперь $G$ --- $(\ast)$-предпучок на $Sm/k$. Тогда в силу $\mathbb{A}^{1}$-инвариантности $G$ для любого $X\in Sm/k$ вложение $\tau:\mathbb{G}_{m}\hookrightarrow\mathbb{A}^{1}$ индуцирует гомоморфизм $\tau^{*}:G(X)\cong F(X\times\mathbb{A}^{1})\to G(X\times\mathbb{G}_{m})$. Мы будем обозначать предпучок $X\to G(X\times\mathbb{G}_{m})$ через $G^{\mathbb{G}_{m}}$, а соответствующий фактор-предпучок $X\to G^{\mathbb{G}_{m}}(X)/G(X)$ --- через $G_{-1}$. Получившиеся предпучки также удовлетворяют $(\ast)$. Если $G$ является пучком Нисневича, то пучком Нисневича также является и $G_{-1}$.
Предпучок $G_{-n}$ мы будем определять как результат применения к $G$ соответствующей операции $n$ раз. Лля любого $(\ast)$-предпучка и $n>0$ имеет место равенство $(G_{Nis})_{-n}=(G_{-n})_{Nis}$ (см. нашу статью, Теорема 4.1).

Следующая теорема позволяет проинтерпретировать представленную выше резольвенту в терминах пучков $(F_{Nis})_{-n}$:

\begin{theorem}\label{the:PTmain}~\cite[Теорема 6.4]{D}
Пусть $F$ --- $(\ast)$-предпучок, $X\in Sm/k$ --- $k$-гладкая неприводимая схема и $i_{x}:x \hookrightarrow X$ --- некоторая точка $X$ коразмерности $d$. Тогда для любого $n\geqslant 0$ имеют место следующие изоморфизмы:
$$
(H^{n}_{Nis})_{x}(X,F_{Nis})=
\begin{cases}

(F_{Nis})_{-d}(Spec(k(x))) & n=d\\
0 & n\neq d\\
\end{cases}  
$$     
\end{theorem}

В частности, представленная выше вялая резольвента приобретает вид
\begin{equation}\label{eq:Gerst}
0\to F_{Nis}\to\underset{x\in X^{(0)}}{\bigoplus}(i_{x})_{\ast}(F_{Nis})\to\dots\to\underset{x\in X^{(m)}}{\bigoplus}(i_{x})_{\ast}((F_{Nis})_{-m})\to 0
\end{equation}

При этом, уже доказанный нами случай $F_{Nis}=F_{Zar}$ позволяет использовать последовательность~\ref{eq:Gerst} как вялую резольвенту для пучка Зарисского $F_{Zar}$. Таким образом, для любого $n>0$ выполняется равенство $H^{n}_{Zar}(X,F_{Zar})=H^{n}_{Nis}(X,F_{Nis})$. Теорема~\ref{the:main} доказана.
\begin{remark} Отметим, что изначально в~\cite{D} теорема 6.4 была сформулирована в предположении, что поле $k$ имеет нулевую характеристику. Тем не менее, 
в актуальном доказательстве условие ненулевой характеристики бесконечного совершенного поля не использовалось за вычетом апелляции к теореме~\cite[Теорема 3.15(5)]{GP1}, которая исключает случай $char(k)=2$. В~\cite[Теорема 3.8]{DP} этот пробел устраняется, и таким образом, доказательство ~\cite[Теорема 6.4]{D} пословно воспроизводится для случая бесконечного совершенного поля $k$.
	\end{remark}

\end{document}